\documentclass[11pt,a4paper]{amsart}

\usepackage{amssymb}

\theoremstyle{plain}
\newtheorem{teo}{Theorem}
\newtheorem{pro}[teo]{Proposition}
\newtheorem{cor}[teo]{Corollary}

\theoremstyle{definition}

\newtheorem{que}{Question}

\theoremstyle{remark}
\newtheorem{rem}[teo]{Remark}

\DeclareMathOperator{\Irr}{Irr}
\DeclareMathOperator{\SL}{SL}
\DeclareMathOperator{\St}{St}
\DeclareMathOperator{\red}{red}
\DeclareMathOperator{\ind}{ind}

\newcommand{\F}{\mathbb{F}}
\newcommand{\N}{\mathbb{N}}
\newcommand{\Z}{\mathbb{Z}}
\newcommand{\Q}{\mathbb{Q}}
\newcommand{\C}{\mathbb{C}}

\renewcommand{\theta}{\vartheta}



\thanks{A.~Jaikin-Zapirain and J.~Gonz\'alez-S\'anchez were supported
  by grant MTM2011-28229-C02-01 from the Spanish Ministry of Economy
  and Competitivity.  J.~Gonz\'alez-S\'anchez also acknowledges support
  through the Ram\'on y Cajal Programme of the Spanish Ministry of
  Science and Innovation.  B.~Klopsch is grateful for financial
  support which he received as a visitor at the Universidad Aut\'onoma
  de Madrid}

\begin{document}

\title[Representation zeta function vanishes at $-2$]{The
  representation zeta function of a FAb compact $p$-adic Lie group
  vanishes at $-2$}

\author{Jon Gonz\'alez-S\'anchez} \address{
  Departamento de Matem\'aticas \\
  Facultad de Ciencia y Tecnolog\'ia \\
  Universidad del Pa\'is Vasco- Euskal Herriko Unibertsitatea\\
  Apartado 644 48080 Bilbao, Spain} \email{jon.gonzalez@ehu.es}

\author{Andrei Jaikin-Zapirain} \address{
  Departamento de Matem\'aticas \\
  Facultad de Ciencias, m\'odulo 17\\
  Universidad Aut\'onoma de Madrid\\
  Ciudad Universitaria de Cantoblanco \\
  28049-Madrid, Spain} \email{andrei.jaikin@uam.es}

\author{Benjamin Klopsch} 
\address{Institut f\"ur Algebra und Geometrie\\ Mathematische 
  Fakult\"at\\ Otto-von-Guericke-Universit\"at Magdeburg\\ 39106
  Magdeburg\\ Germany}
\email{Benjamin.Klopsch@ovgu.de}


\maketitle

\begin{abstract}
  Let $G$ by compact $p$-adic Lie group and suppose that $G$ is FAb,
  i.e., that $H/[H,H]$ is finite for every open subgroup $H$ of~$G$.
  The representation zeta function $\zeta_G(s) = \sum_{\chi \in
    \Irr(G)} \chi(1)^{-s}$ encodes the distribution of continuous
  irreducible complex characters of~$G$.  For $p \geq 3$ it is known
  that $\zeta_G(s)$ defines a meromorphic function on~$\C$.

  Wedderburn's structure theorem for semisimple algebras implies that
  $\zeta_G(-2) = \lvert G \rvert$ for finite~$G$.  We complement this
  classic result by proving that $\zeta_G(-2) = 0$ for infinite~$G$,
  assuming $p \geq 3$.
\end{abstract}

\section{Introduction}
Let $G$ be a finitely generated profinite group, and let $\Irr(G)$
denote the collection of all continuous irreducible complex characters
of~$G$.  We observe that each $\chi \in \Irr(G)$ has finite degree and
for every positive integer $n \in \N$ we put $r_n(G) = \lvert \{ \chi
\in \Irr(G) \mid \chi(1) = n \} \rvert$.  From Jordan's theorem on
finite linear groups in characteristic~$0$
(see~\cite[Theorem~9.2]{We73}) one deduces that $r_n(G)$ is finite for
all $n \in \N$ if and only if $G$ is FAb, i.e., if $H/[H,H]$ is finite
for every open subgroup $H$ of~$G$.

Suppose that $G$ is FAb.  Then the arithmetic sequence $r_n(G)$, $n
\in \N$, is encoded in the Dirichlet generating function
\begin{equation*}
  \zeta_G(s) = \sum_{n=1}^\infty r_n(G) n^{-s} = \sum_{\chi \in
    \Irr(G)} \chi(1)^{-s}
\end{equation*}
which is known as the \emph{representation zeta function} of~$G$.  If
the representation growth of $G$ is polynomially bounded, i.e., if
$\sum_{n=1}^N r_n(G) = O(N^d)$ for some constant $d$, then
$\zeta_G(s)$ defines an analytic function on a non-empty right
half-plane of~$\C$.  Under favourable circumstances, this function
admits a meromorphic continuation, possibly to the entire complex
plane~$\C$.

In recent years representation growth and representation zeta
functions have been investigated for various kinds of groups,
including compact $p$-adic Lie groups; for instance,
see~\cite{JZ06,LaLu08,AKOV12a,AKOV13} or the short introductory
survey~\cite{Kl13}.  An intriguing, but mostly unexplored aspect is
the significance of special values of representation zeta functions.
In particular, one may be curious about the locations of zeros and
poles.  While there is some theoretical understanding of the latter
(see~\cite[Theorem~B]{AKOV13}, and also compare~\cite{dS94} for the
pole spectra of related zeta functions), almost nothing is known about
the former.

In the present paper we establish that $\zeta_G(s)$ vanishes at $s=-2$
for every member $G$ of a certain class of profinite groups, including
all infinite FAb compact $p$-adic Lie groups for $p \geq 3$.  Indeed,
let $G$ be a finitely generated profinite group which is FAb and
virtually pro-$p$ for some prime~$p$.  We say that $G$ has
\textit{rational representation zeta function} (with respect to $p$),
$\mathrm{r.r.z.f.}_{(p)}$ for short, if there exist finitely many positive
integers $m_1, \ldots, m_k$ and rational functions $R_1, \ldots, R_k
\in \Q(X)$ such that
\begin{equation} \label{equ:formula} \zeta_G(s) = \sum_{i=1}^k
  m_i^{-s} R_i(p^{-s}).
\end{equation}
In~\cite{JZ06}, Jaikin-Zapirain proved that, for $p \geq 3$, every FAb
compact $p$-adic Lie group has~$\mathrm{r.r.z.f.}_{(p)}$.  It is
conjectured that the result extends to $2$-adic Lie groups; presently,
it is known that every uniformly powerful pro-$2$ group
has~$\mathrm{r.r.z.f.}_{(2)}$.

There is only a small number of FAb compact $p$-adic Lie groups $G$
for which the representation zeta function $\zeta_G(s)$ has been
computed explicitly; see~\cite{JZ06,AKOV12a,AKOV13}.  By inspection of
the formula given in~\cite[Theorem~7.5]{JZ06}, Motoaki Kurokawa and
Nobushinge Kurokawa noticed that the representation zeta function of
the $p$-adic Lie group $\SL_2(\Z_p)$ has zeros at~$s=-1$ and $s=-2$.
The purpose of the present paper is to explain the zero at $s=-2$
which reflects a more general phenomenon.  

\begin{teo} \label{thm:main} Let $G$ be a FAb profinite group which is
  infinite and virtually a pro\nobreakdash-$p$ group.  If $G$ has
  rational representation zeta function with respect to~$p$ then
  $\zeta_G(-2) = 0$.
\end{teo}

As indicated, using~\cite[Theorem~1.1]{JZ06} we derive the following
corollary.

\begin{cor} \label{cor:main} Let $G$ be a FAb compact $p$-adic Lie
  group and suppose that $p \geq 3$.  If $G$ is infinite then
  $\zeta_G(-2) = 0$.
\end{cor}

\begin{rem} \label{rem:Wedderburn} Wedderburn's structure theorem for
  semisimple algebras implies that $\zeta_G(-2) = \sum_{\chi \in \Irr
    (G)} \chi(1)^2 = |G|$ for every finite group~$G$.  For an infinite
  profinite group $G$ one can evaluate $\zeta_G(s)$ at $s=-2$ only if
  the function defined by the Dirichlet series has a suitable
  continuation.
\end{rem}

\begin{rem}
  The representation functions of compact open subgroups of
  semi\-simple $p$-adic Lie groups, such as $\SL_n(\Z_p)$, occur
  naturally as factors in Euler products for the representation zeta
  functions of arithmetic lattices in semisimple groups, such $\Gamma
  = \SL_n(\Z) \subseteq \SL_n(\mathbb{R})$; see
  \cite[Proposition~1.3]{LaLu08}.  However, since the Euler product
  formula is not valid for $s=-2$, one cannot use
  Corollary~\ref{cor:main} directly to investigate potential
  properties of $\zeta_\Gamma(s)$ at $s=-2$.  For instance, the
  inverse of the Riemann zeta function $\zeta(s)^{-1} =
  \sum_{n=1}^\infty \mu(n) n^{-s}$ satisfies $\zeta(s)^{-1} = \prod_p
  (1-p^{-s})$ and $\zeta(0)=-2$.
\end{rem}

In the next section we prove Theorem~\ref{thm:main} and its corollary,
by considering the $p$-adic limit of $\zeta_G(s)$ at $s=-2$.  We also
offer an alternative proof of Corollary~\ref{cor:main}, which is
closer to the character theoretic set-up in~\cite{JZ06}.  In the last
section we provide further comments and highlight some open problems.

In conclusion, we remark that related questions regarding zeros and
special values of Witten $L$-functions associated to real Lie groups,
in particular to the groups $\mathrm{SU}(2)$ and $\mathrm{SU}(3)$,
have been considered by Kurokawa and Ochiai~\cite{KuOc13} and also by
Min~\cite{Mi13}.


\section{The proofs}

\begin{proof}[Proof of Theorem~\ref{thm:main}]
  Let $m_1, \ldots, m_k \in \N$ and $R_1, \ldots, R_k \in \Q(X)$ such
  that the Dirichlet series $\zeta_G(s) = \sum_{n=1}^\infty r_n(G)
  n^{-s} = \sum_{\chi \in \Irr(G)} \chi(1)^{-s}$
  satisfies~\eqref{equ:formula}.  Then the degrees of the irreducible
  characters of $G$ are of the form $m_i p^r$ with $i \in \{1, \ldots,
  k\}$ and $r \geq 0$.  In particular, for every positive integer $j$
  there are at most finitely many characters $\chi \in \Irr(G)$ with
  $p^j \nmid \chi(1)$.  Consequently, the series $\zeta_G(s)$
  converges, with respect to the $p$-adic topology, at every negative
  integer $-e \in -\N$ to an element in the ring $\Z_p$ of $p$-adic
  integers: we obtain a function
  \[
  \zeta_G^\text{$p$-adic} \colon -\N \rightarrow \Z_p, \quad -e
  \mapsto \sum_{n=1}^\infty r_n(G) n^e = \sum_{\chi \in \Irr(G)}
  \chi(1)^e.
  \]
  For the last equality recall that in the $p$-adic topology every
  converging series converges unconditionally so that its summands can
  be re-arranged freely.

  Equation~\eqref{equ:formula} reflects more than the equality of two
  complex functions: by expansion of the right hand side we obtain a
  Dirichlet series whose coefficients must agree with the defining
  coefficients $r_n(G)$ of the zeta function on the left hand side.
  This implies that for every negative integer $-e$,
  \begin{equation} \label{equ:p-adic-zeta} \zeta_G^\text{$p$-adic}(-e)
    = \sum_{i=1}^k m_i^e R_i(p^e) = \zeta_G(-e).
  \end{equation}
  Consequently, it suffices to prove that $\zeta_G^\text{$p$-adic}(-2)
  = 0$.

  Fix a positive integer~$j$.  As seen above, there are only finitely
  many characters $\chi \in \Irr(G)$ such that $p^j \nmid \chi(1)$.
  We define
  \begin{equation*}
    N_j = \bigcap_{\chi \in \Irr (G),\; p^j \nmid \chi (1)} \ker \chi,
  \end{equation*}
  where each $\ker \chi$ coincides with the kernel of a
  representation affording~$\chi$.  Then $N_j$ is an open normal
  subgroup of~$G$, and
  \begin{equation*}
    \zeta_G^\text{$p$-adic}(-2) = \sum_{\chi \in \Irr(G)} \chi(1)^2 =
    \sum_{\substack{\chi \in \Irr(G) \text{ with} \\  N_j \subseteq
        \ker \chi}}\chi(1)^2 + \sum_{\substack{\chi \in \Irr(G) \text{
          with} \\ N_j \not \subseteq \ker \chi}} \chi (1)^2.
  \end{equation*}
  The first sum is equal to the order of the finite group $G/N_j$,
  while all the terms in the second sum are divisible by~$p^{2j}$.
  Thus
  \begin{equation} \label{equ:take-limit}
    \zeta_G^\text{$p$-adic}(-2) = \lvert G : N_j \rvert + p^{2j} a_j,
  \end{equation} 
  for some $a_j \in \Z_p$. 

  Since $G$ is infinite and virtually a pro-$p$ group, $\lvert G : N_j
  \rvert + p^{2j} a_j \to 0$ in the $p$-adic topology as $j \to
  \infty$.  Thus \eqref{equ:take-limit} yields
  $\zeta_G^\text{$p$-adic}(-2) = 0$.
\end{proof}

Next we give an alternative proof of Corollary~\ref{cor:main}, which
is closer to the set-up in~\cite{JZ06} and does not rely on $p$-adic
limits.

\begin{pro} \label{pro:1} Suppose that $p \geq 3$ and let $N$ be a FAb
  uniformly powerful pro-$p$ group.  Then for every $m \geq 0$,
  \[
  \zeta_{N^{p^m}}(s) = \lvert N : N^{p^m} \rvert \, \zeta_N(s).
  \]
\end{pro}

\begin{proof}
  This is a consequence of the analysis in~\cite[Section~3]{AKOV13}
  of a formula given in~\cite[Corollary~2.13]{JZ06}.
\end{proof}

\begin{pro} \label{pro:2} Suppose that $p \geq 3$ and let $G$ be a FAb
  compact $p$-adic Lie group.  Let $H$ be an open subgroup of $G$.
  Then
  \[
  \zeta_G(-2) = \lvert G : H \lvert \, \zeta_H(-2).
  \]
\end{pro}

\begin{proof}
  Choose an open normal subgroup $N$ of $G$ which is $2$-uniform (in
  the sense of \cite[Section~2]{JZ06}) and contained in~$H$.  We show
  that
  \[
  \zeta_G(-2) = \lvert G : N \lvert \, \zeta_N(-2).
  \]
  The same reasoning yields $\zeta_H(-2) = \lvert H : N \lvert
  \zeta_N(-2)$, and combining the two equations proves the
  proposition.

  We adapt the set-up in~\cite[Sections~5 and~6]{JZ06}.  As in the
  proof of \cite[Theorem~1.1]{JZ06}, we decompose the representation
  zeta function of $G$ as
  \[
  \zeta_G(s) = \sum_{N \leq K \leq G} \sum_{\substack{\theta \in
      \Irr(N) \\ \text{with } \St_G(\theta)=K}} \lvert G : K
  \rvert^{-1-s} f_{(K,N,\theta)}(s) \cdot\theta(s)^{-1},
  \]
  where for each character triple $(K,N,\theta)$ one defines
  \[
  f_{(K,N,\theta)}(s) = \sum_{\chi \in \Irr(K \vert \theta)}
  \left( \frac{\chi(1)}{\theta(1)} \right)^{-s}
  \]
  summing over all $\chi \in \Irr(K)$ such that $\theta$ is a
  component of~$\red_N^G(\chi)$.  We observe that for each character
  triple $(K,N,\theta)$,
  \begin{equation} \label{equ:f(-2)} f_{(K,N,\theta)}(-2) = \sum_{\chi
      \in \Irr(K \vert \theta)} \left( \frac{\chi(1)}{\theta(1)}
    \right)^2 = \frac{\red_N^K(\ind_N^K(\theta))(1)}{\theta(1)} =
    \lvert K : N \rvert.
  \end{equation}

  It is proved in \cite{JZ06} that for each group $K$ with $N \leq K
  \leq G$ the set $\Irr(N)_K = \{ \theta \in \Irr(N) \mid
  \St_G(\theta) = K \}$ can be partitioned into finitely many subsets
  $\Irr(N)_{K,v}$, labelled by $v \in V_K$, such that
  \begin{enumerate}
  \item[(i)] for each $v \in V_K$ and $\theta \in \Irr(N)_{K,v}$,
    \[
    f_{(K,N,\theta)}(s) = f_v(s)
    \]
    depends only on $v$ and
  \item[(ii)] for each $v \in V_K$,
    \[
    g_v(s) = \sum_{\theta \in \Irr(N)_{K,v}} \theta(s)^{-1}
    \]
    is a rational function over $\Q$ in $p^{-s}$.
  \end{enumerate}
  The equations
  \begin{align*}
    \zeta_G(s) & = \sum_{N \leq K \leq G} \; \sum_{v \in V_K} \lvert G
    : K \rvert^{-1-s} f_v(s) g_v(s), \\
    \zeta_N(s) & = \sum_{N \leq K \leq G} \; \sum_{v \in V_K} g_v(s)
  \end{align*}
  combined with \eqref{equ:f(-2)} give
  \begin{align*}
    \zeta_G(-2) & = \sum_{N \leq K \leq G} \; \sum_{v \in V_K} \lvert
    G : K \rvert \lvert K : N \rvert g_v(-2) \\
    & = \lvert G : N \rvert \zeta_N(-2).
  \end{align*}
\end{proof}

\begin{proof}[Second proof of Corollary~\ref{cor:main}]
  By Proposition~\ref{pro:2} it is enough to prove the result for a
  uniformly powerful pro-$p$ group $N$.  By Propositions~\ref{pro:1}
  and \ref{pro:2} we have
 \[
 \lvert N : N^p \rvert \zeta_N(-2) = \zeta_{N^p}(-2) = \lvert N : N^p
 \rvert^{-1} \zeta_N(-2).
 \]
 Since $\lvert N : N^p \rvert > 1$, this implies $\zeta_N(-2) = 0$.
\end{proof}

\section{Open questions}
We highlight three questions which arise naturally from
Theorem~\ref{thm:main}, Corollary~\ref{cor:main} and their proofs.

In view of \eqref{equ:p-adic-zeta} we record the following problem.

\begin{que}
  Let $G$ be a FAb compact $p$-adic analytic group.  What are the
  values of $\zeta_G(s)$ at other negative integers $s = -e$ and is
  there a suitable interpretation of these?
\end{que}

Of course, we would like to extend Corollary~\ref{cor:main} to the
prime~$p=2$.  More generally, one can ask the following.

\begin{que}
  Let $G$ be a FAb profinite group and and suppose that $\zeta_G(s)$
  converges in some right half-plane $\{ s \in \C \mid \textrm{Re}(s)
  > \alpha \}$.  Suppose further that $\zeta_G(s)$ has a meromorphic
  continuation so that $\zeta_G(-2)$ is defined.  Is it true that
  $\zeta_G(-2) = 0$?
\end{que}

For instance it would be natural to investigate this question for
compact analytic groups over compact discrete valuation rings of
positive characteristic, e.g., over~$\F_p[\![t]\!]$.  The
representation zeta functions of such groups are still rather poorly
understood.  In particular, no analogue of Proposition~\ref{pro:1} is
known.  However, a direct computation in~\cite{JZ06} shows that, for
$p \geq 3$, the group $\SL_2(\F_p[\![t]\!])$ has the same
representation zeta function as the $p$-adic analytic group
$\SL_2(\Z_p)$.

The last question is inspired by Brauer's Problem~$1$, which asks:
what are the possible degree patterns for irreducible characters of
finite groups; see~\cite{Br63,Hu91,Mo07}.  Given a profinite group $G$
with $\mathrm{r.r.z.f.}_{(p)}$, the completed group algebra $\C[\![G]\!] =
\varprojlim_{N \trianglelefteq G} \C[G/N]$, formed with respect to the
directed set of normal open subgroups of $G$, determines the
representation zeta function of~$G$ and, conversely, $\zeta_G(s)$
determines $\C[\![G]\!]$.  Furthermore, if $G$ is a pro-$p$ group,
then $\zeta_G(s)$ is a rational function over $\Q$ in $p^{-s}$.  The
following can be regarded as an extension of Brauer's Problem~$1$ to
FAb pro-$p$ groups.

\begin{que}
  Which rational functions $R(p^{-s})$ over $\Q$ in $p^{-s}$ are
  representation zeta functions of infinite FAb pro-$p$ groups with
  $\mathrm{r.r.z.f.}_{(p)}$?
\end{que}

Theorem~\ref{thm:main} provides a first necessary criterion:
$R(p^2)=0$.


\end{document}